\documentclass{amsart}

\usepackage{amsmath,arydshln,multirow}

\usepackage{hyperref}

\usepackage{cases}
\usepackage{amsmath}
\usepackage{amsfonts}
\usepackage{bm}
\usepackage{arydshln}
\usepackage{amsfonts,amsmath,amssymb,amscd,bbm,amsthm,mathrsfs,dsfont}
\usepackage{mathrsfs}
\usepackage{pb-diagram}
\usepackage{amssymb}
\usepackage{tikz-cd}
\usepackage{enumitem}
\usepackage[all]{xy}
\newtheorem{Thm}{Theorem}[section]

\newtheorem{Prop}[Thm]{Proposition}
\newtheorem{Lem}[Thm]{Lemma}

\newtheorem{Def}[Thm]{Definition}
\newtheorem{example}[Thm]{Example}
\newtheorem{proposition-definition}[Thm]{Proposition-Definition}

\newtheorem{conj}[Thm]{Conjecture}

\theoremstyle{remark}
\newtheorem{Rem}[Thm]{Remark}

\newtheorem{Question}[Thm]{Question}
\newtheorem{Convention}[Thm]{Convention}
\numberwithin{equation}{section}

\usepackage[OT2,T1]{fontenc}
\DeclareSymbolFont{cyrletters}{OT2}{wncyr}{m}{n}
\DeclareMathSymbol{\Sha}{\mathalpha}{cyrletters}{"58}
\newcommand{\et}{\mathrm{\acute{e}t}}

\newcommand{\Bcyr}{\text{\fontencoding{OT2}\selectfont B}}

\begin{document}

\title[On the second partial Global Euler--Poincar\'e characteristics]{On the second partial Global Euler--Poincar\'e characteristics for Galois cohomology}


	
	\author{Yufan Luo}
	
	\subjclass[2020]{}
	
	\dedicatory{}
	
	\address{Shanghai Institute for Mathematics and Interdisciplinary Sciences (SIMIS), Shanghai 200433, China
	}
	
	\address{Research Institute of Intelligent Complex Systems, Fudan University, Shanghai
		200433, China}
	\email{yufanluo@hotmail.com}

	\maketitle
	\begin{abstract}
		Let $K$ be a number field, let $S$ be a finite set of primes of $K$ containing all archimedean primes, and let $G_{K,S}$ denote the Galois group of the maximal extension of $K$ unramified outside $S$. In this paper, we study the second partial Euler--Poincar\'e characteristic $\chi_{2}(G_{K,S},M)$ for a finite $G_{K,S}$-module $M$, without imposing the condition that the order of $M$ is an $S$-unit. By adjoining a further finite set of primes of $K$, which can be chosen to be disjoint from any prescribed set of primes of density zero, we obtain an explicit formula for the corresponding second partial Euler--Poincar\'e characteristic. As an application, we investigate the presentation of the Galois group $G_{K,S}$. Furthermore, for any number field, we construct counterexamples to the dimension conjecture for Galois deformation rings.
	\end{abstract} 
	
\section{Introduction}
\subsection{Global Euler--Poincar\'e characteristics}
Let $K$ be a number field, $S$ a finite set of primes of $K$ containing all archimedean primes, and $G_{K,S}$ the Galois group of the maximal extension of $K$ unramified outside $S$. For $v \in S$, let $K_{v}$ be the completion of $K$ at $v$. Let $\overline{K}$ (resp.\ $\overline{K_{v}}$) be an algebraic closure of $K$ (resp.\ $K_{v}$), and let $G_{K} := \operatorname{Gal}(\overline{K}/K)$ (resp.\ $G_{v} := \operatorname{Gal}(\overline{K_{v}}/K_{v})$) be the absolute Galois group of $K$ (resp.\ $K_{v}$).

We fix a homomorphism $\overline{K} \hookrightarrow \overline{K_{v}}$ extending $K \hookrightarrow K_{v}$, yielding a homomorphism $G_{v} \to G_{K} \to G_{K,S}$. Let $M$ be a finite discrete $G_{K,S}$-module. Recall that the second partial Euler--Poincar\'e characteristic $\chi_{2}(G_{K,S},M)$ for the $G_{K,S}$-module $M$ is defined by
\[ \chi_{2}(G_{K,S},M) := \frac{[H^{0}(G_{K,S},M)] \, [H^{2}(G_{K,S},M)]}{[H^{1}(G_{K,S},M)]}, \]
where $[-]$ denotes cardinality. Assume that the cardinality $[M]$ of $M$ is an $S$-unit, i.e., $S$ contains all the primes dividing the order of $M$. Then Tate's global Euler--Poincar\'e characteristic formula (see \cite[Chapter I, Theorem 5.1]{MR2261462}) states that 
\begin{equation}\label{usualeuler}
	\chi_{2}(G_{K,S},M) = \frac{1}{[M]^{[K:\mathbb{Q}]}} \prod_{v \in S_{\infty}} [H^{0}(G_{v},M)], 
\end{equation}
where the product is over all archimedean primes of $K$. This naturally leads to the following question:

\begin{Question}\label{questionA}
	If $[M]$ is not an $S$-unit, how can one calculate $\chi_{2}(G_{K,S},M)$?
\end{Question}
 
Our first observation is that if $M = \bigoplus_{p} M(p)$ is the canonical decomposition of $M$ into $p$-primary components, then we have $H^{i}(G_{K,S},M) = \bigoplus_{p} H^{i}(G_{K,S},M(p))$ for any $i \geq 0$, and hence $\chi_{2}(G_{K,S},M) = \prod_{p} \chi_{2}(G_{K,S},M(p))$. Without loss of generality, we may assume that $M$ is $p$-primary. 

The main theorem of this paper, stated below, provides a partial answer to Question \ref{questionA}.

	   \begin{Thm}\label{main}
	   Let $K$ be a number field, $p$ a prime number and $S$ a finite set of primes of $K$ containing all archimedean primes. Let $M$ be a finite discrete $p$-primary $G_{K,S}$-module. Then the following assertions hold.
	   \begin{enumerate}
	   	\item We have
	   	\begin{equation}\label{euler}
	   		\chi_{2}(G_{K,S},M)\leq \dfrac{1}{[M]^{[K:\mathbb{Q}]}}\prod_{v\in S_{\infty}}[H^{0}(G_{v},M)] \cdot
	   		\prod_{v\notin S, v\in S_{p}} \dfrac{1}{|[M]|_{v}}\cdot \epsilon,
	   	\end{equation}
	    where $S_p$ is the set of primes of $K$ above $p$, $S_{\infty}$ is the set of all archimedean primes of $K$, $|-|_{v}$ is the normalized absolute value associated to the prime $v$, and
	  	\[ 
	  \epsilon := \begin{cases} 
	  	1, & \text{if } S_{\mathrm{fin}} \neq \emptyset, \\ 
	  	[(M')^{G_{K}}], & \text{if } S_{\mathrm{fin}} = \emptyset, \text{ with } p\neq 2 \text{ or } K \text{ totally imaginary}, \\ 
	  	\left[ \bigcap_{v\in S_{\mathbb{R}}} \left( (1+c_{v})M' \cap (M')^{G_{K}} \right) \right], & \text{if } S_{\mathrm{fin}} = \emptyset, \text{ with } p=2 \text{ and } K \text{ not totally imaginary}. 
	  \end{cases}
	  \]
	  	Here, $S_{\mathrm{fin}} \subset S$ denotes the subset of nonarchimedean primes, $S_{\mathbb{R}} \subset S_{\infty}$ is the subset of real primes, and $c_{v} \in G_{K}$ is a chosen complex conjugation for each $v\in S_{\mathbb{R}}$. The module $M' := \operatorname{Hom}(M,\overline{K}^{\times})$ is the Cartier dual of $M$, viewed as a $G_{K}$-module via the canonical quotient $G_{K} \twoheadrightarrow G_{K,S}$.
        
        \item Let $\mathcal{O}_{K,S}$ be the ring of $S$-integers in $K$ and $\mathcal{M}$ be the corresponding locally constant \'etale sheaf on the affine scheme $\operatorname{Spec}(\mathcal{O}_{K,S})$. Then the equality \eqref{euler} holds if and only if the natural morphism
        \begin{equation}\label{phi2}
        	\phi_{2}: H^{2}(G_{K,S}, M) \to H^{2}_{\et}(\operatorname{Spec}(\mathcal{O}_{K,S}),\mathcal{M})
        \end{equation}
        is an isomorphism, where $H^{2}_{\et}(\operatorname{Spec}(\mathcal{O}_{K,S}), \mathcal{M})$ denotes the second \'etale cohomology group of the sheaf $\mathcal{M}$.

	   	\item Assume that $p\neq 2$ or that $K$ is totally imaginary. Suppose that $\mathcal{T}$ is a set of primes of $K$ of Dirichlet density zero such that $\mathcal{T}\cap S=\emptyset$. Then there exists a finite set $S_{0}$ of primes of $K$, disjoint from $S\cup \mathcal{T}$, such that the set of nonarchimedean primes in $S\cup S_{0}$ is nonempty and, when $M$ is viewed as a $G_{K,S\cup S_{0}}$-module, the equality in \eqref{euler} holds, i.e. 
	   	\begin{equation}\label{nice}
	   		\chi_{2}(G_{K,S\cup S_{0}},M)= \dfrac{1}{[M]^{[K:\mathbb{Q}]}}\prod_{v\in S_{\infty}}[H^{0}(G_{v},M)]\cdot
	   		\prod_{v\notin S\cup S_{0}, v\in S_{p}} \dfrac{1}{|[M]|_{v}}.
	   	\end{equation}
	   In particular, if $S\cap S_p=\emptyset$, then one may take $\mathcal{T}=S_p$, in which case \eqref{nice} reduces to
	   		\[	\chi_{2}(G_{K,S\cup S_{0}},M)= \prod_{v\in S_{\infty}}[H^0(G_v,M)].  \]
	   \end{enumerate}
	   \end{Thm}
	  
	    \begin{Rem}
	   	\begin{enumerate}
	   		\item In general, examples can easily be constructed in which the inequality \eqref{euler} is strict; see Example \ref{anexample}. 
	   		\item Note that if $S$ contains all primes above $p$, then $\phi_{2}$ in \eqref{phi2} is an isomorphism by \cite[Proposition 2.9, Chapter II]{MR2261462} and hence the inequality \eqref{euler} is actually an equality, which coincides exactly with \eqref{usualeuler}. Therefore, our theorem can be viewed as a generalization of Tate's global Euler--Poincar\'e characteristic formula.
	   		\item When $M$ is a finite-dimensional vector space over the finite field $\mathbb{F}_{p}$ of order $p$, Liu established a weaker version of assertion (1) in \cite[Proposition 9.4]{MR4896734}.
	   	\end{enumerate}
	   \end{Rem}
	   
	  This theorem provides the optimal upper bound for the general formula of the second partial Euler--Poincar\'e characteristic $\chi_{2}(G_{K,S},M)$. Moreover, it shows that the upper bound can be attained by suitably enlarging the set of primes by a finite set, which can be chosen to be disjoint from any prescribed set of primes of density zero, thereby yielding an explicit formula.
	   
	  The proof strategy of the theorem involves a comparison between the \'etale cohomology groups and the Galois cohomology groups, along with an exact computation of the \'etale cohomology groups. Furthermore, the third part of the theorem is established by refining the techniques of Schmidt \cite{MR2365909}.
	  
   \subsection{The presentation of the group $G_{K,S}$}
   In this subsection, we present an application of Theorem \ref{main} to the presentation of the profinite group $G_{K,S}$.
   If $G$ is a topologically finitely generated profinite group, then we denote by $d(G)$ the minimal number of topological generators of $G$ and by $r(G)$ the minimal number of relations in a profinite presentation of $G$. We establish the following theorem, refining the result of \cite[Theorem 1.2]{MR4896734}.

   \begin{Thm}\label{presentation}
   	Let $K$ be a number field and $S$ a finite set of primes of $K$ containing all archimedean primes. Suppose that the group $G_{K,S}$ is topologically finitely generated. Then the following assertions hold.
   	\begin{enumerate}
   		\item We have
   		\begin{equation}\label{numberof}
   	  0\leq r(G_{K,S})-d(G_{K,S})\leq \mathcal{R}(K)-\gamma,
   		\end{equation}
   		where $\mathcal{R}(K)$ denotes the number of archimedean primes of $K$ and
   		\[ \gamma:=\begin{cases}
   			0,&\text{if $S_{\mathrm{fin}}=\emptyset$},\\
   			1,&\text{if $S_{\mathrm{fin}}\neq \emptyset$}.
   		\end{cases} \]
   		Furthermore, the profinite group $G_{K,S}$ admits a finite presentation on $d(G_{K,S})$ generators and $\mathcal{R}(K)+d(G_{K,S})-\gamma$ relations.
   		\item There exists a finite set $S_{0}$ of nonarchimedean primes of $K$ such that 
   		\[ r(G_{K,S\cup S_{0}})-d(G_{K,S\cup S_{0}})=\mathcal{R}(K)-1.\]
   	\end{enumerate}
   \end{Thm}
    It is easy to construct examples for which the inequality in \eqref{numberof} is strict; see Example \ref{anexample}. The proof of this theorem will be established by combining the work of Lubotzky in \cite{MR1848964} with our main Theorem \ref{main}. 
   
	\subsection{Mazur's dimension conjecture} 
	In this subsection, we present applications of our main theorem to Galois deformation theory. Let $p$ be a prime, $K$ a number field, and $S$ a finite set of primes of $K$. Let $\mathbb{F}$ be a finite field of characteristic $p$. Let $\overline{\rho}: G_{K,S} \to \mathrm{GL}_{n}(\mathbb{F})$ be a continuous, absolutely irreducible representation. Let $W(\mathbb{F})$ denote the ring of Witt vectors of $\mathbb{F}$. If $(A, \mathfrak{m})$ is a complete Noetherian local $W(\mathbb{F})$-algebra with residue field $\mathbb{F}$, then a deformation $\rho$ of $\overline{\rho}$ to $A$ consists of an equivalence class of homomorphisms $\rho: G_{K,S} \to \mathrm{GL}_{n}(A)$ such that the composition of $\rho$ with the natural projection $\mathrm{GL}_{n}(A) \to \mathrm{GL}_{n}(A/\mathfrak{m}) = \mathrm{GL}_{n}(\mathbb{F})$ is $\overline{\rho}$. Here, two lifts $\rho_{1}$ and $\rho_{2}$ of $\overline{\rho}$ to $A$ are said to be equivalent if $\rho_{1} = B\rho_{2}B^{-1}$ for some $B \in \mathrm{GL}_{n}(A)$ congruent to the identity matrix modulo $\mathfrak{m}$. In \cite{MR1012172}, Mazur proved that a universal deformation ring $R_{\overline{\rho}}$ of $\overline{\rho}$ and a universal deformation $\rho^{\mathrm{univ}}: G_{K,S} \to \mathrm{GL}_{n}(R_{\overline{\rho}})$ exist. In \cite[Section 1.10]{MR1012172}, Mazur proposed the following conjecture regarding the Krull dimension of this universal deformation ring.
	
\begin{conj}[Mazur's dimension conjecture]\label{mazdimension}
	If $S$ contains $S_{p}$ and $S_{\infty}$, then the universal deformation ring $R_{\overline{\rho}}$ of $\overline{\rho}$ satisfies 
	\[ \mathrm{Krulldim}(R_{\overline{\rho}}/(p)) = h^{1}(G_{K,S}, \mathrm{ad}) - h^{2}(G_{K,S}, \mathrm{ad}), \]
	where $h^{i}(-) := \dim_{\mathbb{F}} H^{i}(-)$ for any $i$, and $\mathrm{ad} := \mathrm{ad}(\overline{\rho})$ is the adjoint representation of $\overline{\rho}$.
\end{conj}

In \cite[Conjecture 1]{MR1357212} and \cite[Lecture 4]{MR1860043}, Gouv\^ea proposed an extension of Conjecture \ref{mazdimension} to broader contexts.

\begin{conj}\label{dimension}
	Conjecture \ref{mazdimension} holds without the condition that $S \supset S_{p} \cup S_{\infty}$.
\end{conj}

By \cite[Theorem 4.2]{MR1860043}, it is always true that 
\[ \mathrm{Krulldim}(R_{\overline{\rho}}/(p)) \geq h^{1}(G_{K,S}, \mathrm{ad}) - h^{2}(G_{K,S}, \mathrm{ad}),\]
and hence the conjecture is equivalent to the claim that
\begin{equation}\label{mazurequa}
	\mathrm{Krulldim}(R_{\overline{\rho}}/(p)) \leq h^{1}(G_{K,S}, \mathrm{ad}) - h^{2}(G_{K,S}, \mathrm{ad}).
\end{equation}

In \cite[Theorem 1.1]{MR2285736}, Bleher and Chinburg proved that there exist infinitely many real quadratic fields $K$ admitting a mod $2$ representation of $G_{K,\emptyset}$ whose universal deformation ring does not satisfy \eqref{mazurequa}. In the following, we show that Conjecture \ref{dimension} fails for all number fields.

\begin{Thm}\label{counterexampletodimensionalconj}
	Let $p$ be a prime, $K$ a number field, and $S$ a finite set of primes of $K$. Assume that $p \neq 2$ or that $K$ is totally imaginary. Assume further that $S \cap S_{p} = \emptyset$. Let $\overline{\rho}: G_{K,S} \to \mathrm{GL}_{n}(\mathbb{F})$ be a continuous, absolutely irreducible representation, where $n \geq 2$. Then there exists a finite set $T \supset S$ such that $T \cap S_p = \emptyset$ and $h^{1}(G_{K,T}, \mathrm{ad}) - h^{2}(G_{K,T}, \mathrm{ad}) < 0$. In particular, Conjecture \ref{dimension} fails in this case.
\end{Thm}

The theorem indicates that Conjecture \ref{dimension} frequently fails when $S \cap S_{p} = \emptyset$. This is consistent with the prediction of the celebrated unramified Fontaine--Mazur conjecture. Specifically, we have the following theorem:

\begin{Thm}\label{corofUFM}
	Assume the unramified Fontaine--Mazur conjecture \cite[Conjecture 5a]{MR1363495} holds. Suppose $n \geq 2$ is an integer and $p > 2n^{2} - 1$. Let $K$ be a totally real number field, and let $S$ be a finite set of primes of $K$ such that $S \cap S_{p} = \emptyset$. If $\overline{\rho}: G_{K,S} \to \mathrm{GL}_{n}(\mathbb{F})$ is a continuous representation whose image contains $\mathrm{SL}_{n}(\mathbb{F})$, then we have $h^{1}(G_{K,S}, \mathrm{ad}) - h^{2}(G_{K,S}, \mathrm{ad}) < 0$.
\end{Thm}

	\subsection{Organization and Notation} This paper is organized as follows. In Section \ref{calculationofetalecohomology}, we establish the second partial Euler--Poincar\'e characteristic for \'etale cohomology groups of number fields. Section \ref{proofofmaintheorem} is devoted to the proof of Theorem \ref{main}. Theorem \ref{presentation} is proved in Section \ref{proofofsecondtheorems}. An example is given in Section \ref{examples}. Finally, Section \ref{proofofothertheorems} contains the proofs of Theorems \ref{counterexampletodimensionalconj} and \ref{corofUFM}.
	
	Throughout this paper, $ p $ will be a prime number, $ \mathbb{Z} $ will be the ring of integers, and $ \mathbb{Q} $ will be the field of rational numbers. Let $ \mathbb{F}_{p} $ be the finite field of order $ p $. If $F$ is a field, then we denote by $\overline{F}$ an algebraic closure of $F$. If $M$ is a finite group, then we denote by $[M]$ the cardinality of $M$.
	
	If $ K $ is a number field, then we denote by $ S_{p}=S_{p}(K) $ the set of primes of $ K $ above $ p $, by $ S_{\infty}=S_{\infty}(K) $ the set of all archimedean primes of $ K $ and by $S_{\mathbb{R}}$ the subset of $S_{\infty}$ consisting of real primes. If $ S $ is a set of primes of $ K $, then  we denote by $S_{\mathrm{fin}}\subset S$ the subset of all nonarchimedean primes of $K$ in $S $. Also, we denote by $K_{S}$ (resp. $K_{S}(p)$) the maximal extension (resp. maximal $p$-extension) of $K$ inside $\overline{K}$ which is unramified outside $S$. Put $G_{K,S}:=\operatorname{Gal}(K_{S}/K)$ and $G_{K,S}(p):=\operatorname{Gal}(K_{S}(p)/K)$. Also, we denote by $\mathcal{O}_{K,S}:=\{a\in K \mid \mathrm{ord}_{v}(a)\geq 0~\text{for all}~v\notin S\}$ the ring of $S$-integers in $K$ if $S$ contains $S_{\infty}(K)$ where $ \mathrm{ord}_{v}$ is the valuation associated with $v$. Let $\mathcal{R}(K)$ denote the number of archimedean primes of $K$. For any prime $v$ of $K$, we denote by $G_{v}$ the absolute Galois group of the completion $K_{v}$ of $K$ at $v$. 
	
	 If $M$ is a finite $G_{k}$-module where $G_{k}:=\operatorname{Gal}(k^{s}/k)$ is the absolute Galois group of a field $k$, then we denote by $M':=\operatorname{Hom}(M, {k^s}^{\times})$ the Cartier dual of the $G_{k}$-module $M$ where the action of $G_{k}$ on $M'$ is defined by $(g\cdot \varphi)(x):=g(\varphi(g^{-1}\cdot x))$ for $g \in G_{k},\varphi\in M',x\in M$.

	\section{Calculation of the second partial Euler--Poincar\'e characteristics for \'etale cohomology}\label{calculationofetalecohomology}
	
	Our goal in this section is to calculate the second partial Euler--Poincar\'e characteristic for \'etale cohomology groups of number fields. Let $K$ be a number field and $S$ be a finite set of primes of $K$ containing all archimedean primes. Put $U:=\operatorname{Spec}(\mathcal{O}_{K,S})$ and $\eta:=\operatorname{Spec}(K)$. Then we have $\pi_1(U,\bar{\eta})=G_{K,S}$ where $\bar{\eta}:\operatorname{Spec}(\overline{K})\to U$ is a geometric generic point and $\pi_{1}$ denotes the \'etale fundamental group, and there is an equivalence between the category of locally constant constructible \'etale sheaves on $U$ and the category of finitely generated discrete $G_{K,S}$-modules. Let $M$ be a finite discrete $G_{K,S}$-module, and let $\mathcal{M}$ be the corresponding locally constant constructible sheaf on $U$.

	   If $v$ is a prime of $K$, then we will denote by $K_{v}$ the completion of $K$ at $v$ and $\eta_{v}:=\operatorname{Spec}(K_{v})$ via the composite $\eta_{v}\to \operatorname{Spec}(K)\to U$ by $(-)_{\eta_{v}}$. Denote by $R\Gamma(U,-)$ (resp. $R\Gamma(K_{v},(-)_{\eta_{v}})$) the \'etale cohomology on $U$ (resp. $\operatorname{Spec}(K_{v})$), and  denote by $H^{\ast}_{\et}(U,\mathcal{M})$ (resp. $H_{\et}^{\ast}(K_{v},\mathcal{M})$) the cohomology groups of the complex $R\Gamma(U,\mathcal{M})$ (resp. $R\Gamma(K_{v},\mathcal{M})$). Note that here $R\Gamma(K_{v},-)$ is just the Galois cohomology of the absolute Galois group $G_{v}$ of $K_{v}$. For any prime $v$ of $K$, we have a natural morphism
	   \[ R\Gamma(U,-)\to R\Gamma(K_{v},(-)_{\eta_{v}}). \]
	   We define the \'etale cohomology groups with compact support $H^{\ast}_{c}(U,\mathcal{M})$ of the sheaf $\mathcal{M}$ on $U$ to be the cohomology groups of the complex $R\Gamma_{c}(U,\mathcal{M})$ which is defined by the fiber sequence
	   \[ R\Gamma_{c}(U,\mathcal{M})\to R\Gamma(U,\mathcal{M})\to \prod_{v\in S}R\Gamma(K_{v},\mathcal{M}). \]
	   By \cite[Proposition 2.1]{MR4502240}, there is a long exact sequence of groups
	   \begin{align}\label{qiechu}
	   	 \cdots \to H^{i}_{c}(U,\mathcal{M})\to H_{\et}^{i}(U,\mathcal{M})\to \prod_{v\in S}H_{\et}^{i}(K_{v},\mathcal{M})\to H^{i+1}_{c}(U,\mathcal{M})\to \cdots 
	   \end{align}
	   
	   \begin{Prop}\label{gaoweitrivial}
	   With the notation as above, the cohomology groups $H^{i}_{c}(U,\mathcal{M})$ vanish for all $i\neq 0,1,2,3$. Moreover, $H^{i}_{c}(U,\mathcal{M})$ and $H^{i}_{\et}(U,\mathcal{M})$ are finite for all $i=0,1,2,3$. 
	   \end{Prop}
	   \begin{proof}
	   	By \cite[Theorem 3.2]{MR4502240}, the cohomology groups $H^{i}_{c}(U,\mathcal{M})$ are $0$ for $i\neq 0,1,2,3$ and finite for $i=0,1,2,3$. For any $v\in S$ and any $i$, it follows from \cite[Theorem~7.1.8]{MR2392026} that the group $H^{i}_{\et}(K_{v},\mathcal{M})\cong H^{i}(G_{v},M)$ is finite. Since $S$ is a finite set, it follows from (\ref{qiechu}) that the cohomology groups $H^{i}_{\et}(U,\mathcal{M})$ are finite for all $i=0,1,2,3$.
	   \end{proof}
	   
	 Therefore, we can define
	   \[ \chi_{c}(U,\mathcal{M}):=\dfrac{[H^{0}_{c}(U,\mathcal{M})][H_{c}^{2}(U,\mathcal{M})]}{[H_{c}^{1}(U,\mathcal{M})][H_{c}^{3}(U,\mathcal{M})]},\]
	   and
	   \[ \chi_{3}(U,\mathcal{M}):=\dfrac{[H^{0}_{\et}(U,\mathcal{M})][H_{\et}^{2}(U,\mathcal{M})]}{[H_{\et}^{1}(U,\mathcal{M})][H_{\et}^{3}(U,\mathcal{M})]}. \]
	   
	   The following theorem is due to Adrien Morin.
	  \begin{Thm}\label{eulercharacteristicformulaetale}
	With the notation as above, we have $\chi_{c}(U,\mathcal{M})=1$ and
	\[	\chi_{3}(U,\mathcal{M})= \dfrac{1}{[M]^{[K:\mathbb{Q}]}}
	\prod_{v\notin S \cup S_{\infty}} 1/|[M]|_v \cdot \prod_{v\in S_{\infty}} \frac{[H^0(G_v,M)]}{[\widehat{H}^0(G_{v},M)]}, \]
	where $|~|_{v}$ is the normalized absolute value associated to the prime $v$, \footnote{That is, $|x|_{v}=q^{-\operatorname{ord}_{v}(x)}$ where $q$ is the cardinality of the residue field of $v$ and $\operatorname{ord}_{v}$ is the additive valuation with value group $\mathbb{Z}$, cf. \cite[Section 1, Chapter III]{MR1697859}.} and $\widehat{H}^{0}(G_{v},M)$ denotes the $0$-th Tate cohomology of the $G_{v}$-module $M$.
	\end{Thm}
	\begin{proof}
	It follows from \cite[Proposition 6.23]{MR4502240} that $\chi_c(U,\mathcal{M})=1$. Moreover, combining Proposition \ref{gaoweitrivial} with (\ref{qiechu}), we obtain
		\begin{equation}\label{11}
			\chi_{3}(U,\mathcal{M})=\prod_{v\in S}\prod^{3}_{i=0}[H_{\et}^{i}(K_{v},\mathcal{M})]^{(-1)^{i}}.
		\end{equation}
		
	For each $v\in S$, we put $\chi(K_{v},\mathcal{M}):=\prod_{i=0}^{3}[H^{i}_{\et}(K_{v},\mathcal{M})]^{(-1)^{i}}$. By \cite[Theorem 7.1.8, Theorem 7.3.1 and Proposition 1.7.6]{MR2392026}, we have
	\begin{equation}\label{22}
	\chi(K_{v},\mathcal{M})=\begin{cases}
		|[M]|_{v},& v\in S_{\mathrm{fin}},\\
		\dfrac{[H^{0}(G_{v},M)]}{[\widehat{H}^{0}(G_{v},M)]},& v\in S_{\infty}.
	\end{cases}
	\end{equation}
	
	Furthermore, the product formula \cite[Chapter III, Proposition 1.3]{MR1697859} implies that
	\begin{equation}\label{33}
		\prod_{v\in S_{\mathrm{fin}}}|[M]|_{v}= \prod_{v\in S_{\infty}}1/|[M]|_{v} \cdot \prod_{v\notin S\cup S_{\infty}} 1/|[M]|_v  =\dfrac{1}{[M]^{[K:\mathbb{Q}]}} \prod_{v\notin S\cup S_{\infty}} 1/|[M]|_v.
	\end{equation}

   	Substituting (\ref{22}) and (\ref{33}) into (\ref{11}), we deduce that 
		\begin{align*}
			\chi_{3}(U,\mathcal{M})=& \prod_{v\in S}\prod^{3}_{i=0}[H_{\et}^{i}(K_{v},\mathcal{M})]^{(-1)^{i}} \\
			 =& \prod_{v \in S_{\mathrm{fin}}} |[M]|_v \cdot \prod_{v\in S_{\infty}}\dfrac{[H^{0}(G_{v},M)]}{[\widehat{H}^{0}(G_{v},M)]}\\
			 =& \dfrac{1}{[M]^{[K:\mathbb{Q}]}}.\prod_{v\notin S\cup S_{\infty}} 1/|[M]|_v \cdot \prod_{v\in S_{\infty}} \frac{[H^0(G_v,M)]}{[\widehat{H}^0(G_v,M)]},
		\end{align*}
			which completes the proof.
	\end{proof}
	
	The above theorem generalizes \cite[Proposition~3.2]{MR2365909}. To pass from the third characteristic $\chi_3(U,\mathcal{M})$ to the second partial characteristic $\chi_2(U,\mathcal{M})$, we need an exact description of the third cohomology group with compact support as follows.
	
	\begin{Lem}\label{calculationofH3}
	With the notation as above, we have
		\begin{equation*}
			H^{3}_{c}(U,\mathcal{M})^{\vee}\cong (M')^{G_{K}},
		\end{equation*}
		where $(\cdot)^{\vee}$ denotes the Pontryagin dual, $M'$ is the Cartier dual of $M$, and $M$ is viewed as a $G_{K}$-module via the canonical quotient $G_{K} \twoheadrightarrow G_{K,S}$.
	\end{Lem}
	\begin{proof}
		  By Artin--Verdier duality \cite[Theorem 3.2 and Proposition 2.3]{MR4502240}, we have
		\begin{equation*}
			H^{3}_{c}(U,\mathcal{M})^{\vee}\cong \operatorname{Hom}_{U}(\mathcal{M},\mathbb{G}_{m}).
		\end{equation*}
		where $\mathbb{G}_{m}$ is the multiplicative group sheaf on $U$. Since $M$ is a finite discrete $G_{K,S}$-module, its $G_{K,S}$-action factors through a finite quotient group. Let $L/K$ be the finite Galois extension inside $K_S$ that trivializes the action on $M$, and let $G = \operatorname{Gal}(L/K)$. Let $V = \operatorname{Spec} \mathcal{O}_{L,S}$. The morphism $\pi: V \to U$ is a finite \'etale Galois cover with Galois group $G$. 
		
		Let $\mathcal{H}om(\mathcal{M}, \mathbb{G}_m)$ be the internal hom sheaf on the \'etale site of $U$. By definition, evaluating global sections yields $\operatorname{Hom}_{U}(\mathcal{M}, \mathbb{G}_m) = H^0(U, \mathcal{H}om(\mathcal{M}, \mathbb{G}_m))$. By the fundamental property of Galois descent for \'etale sheaves (cf. \cite[Proposition 1.4, Chapter II]{MR559531}), we have
		\[
		\operatorname{Hom}_{U}(\mathcal{M}, \mathbb{G}_m) \cong \operatorname{Hom}_{V}(\pi^*\mathcal{M}, \pi^*\mathbb{G}_m)^G = \operatorname{Hom}_{V}(\underline{M}_V, \mathbb{G}_m)^G,
		\]
		where $\pi^*\mathcal{M}$ becomes the constant sheaf $\underline{M}_V$ on $V$ by the choice of $L$. Since $V$ is connected, we have a canonical isomorphism
		\[
		\operatorname{Hom}_{V}(\underline{M}_V, \mathbb{G}_m) \cong \operatorname{Hom}_{\text{Group}}(M, \mathbb{G}_m(V)) = \operatorname{Hom}_{\text{Group}}(M, \mathcal{O}_{L,S}^\times).
		\]
		Taking $G$-invariants on both sides gives
		\[
		\operatorname{Hom}_{U}(\mathcal{M}, \mathbb{G}_m) \cong \operatorname{Hom}_{G}(M, \mathcal{O}_{L,S}^\times).
		\]
		
		Now, since $M$ is a finite abelian group, any group homomorphism out of $M$ must map into the torsion subgroup of the target. The torsion subgroup of the $S$-units $\mathcal{O}_{L,S}^\times$ is identical to the group of roots of unity $\mu(L)$ in $L$, which implies that
		\[
		\operatorname{Hom}_G(M, \mathcal{O}_{L,S}^\times) \cong \operatorname{Hom}_G(M, L^\times).
		\]
		
		Finally, a $G$-equivariant homomorphism from $M$ to $L^\times$ is canonically equivalent to a $G_K$-equivariant homomorphism from $M$ to $\overline{K}^\times$, because the absolute Galois group $G_L = \operatorname{Gal}(\overline{K}/L)$ acts trivially on $M$, and its invariant subgroup in $\overline{K}^\times$ is exactly $L^\times$. Thus
		\[
		\operatorname{Hom}_G(M, L^\times) \cong (\operatorname{Hom}(M,\overline{K}^\times))^{G_K}.
		\]
		This completes the proof.
	\end{proof}
	
	We are now ready to establish the main result of this section. Let us define the second partial Euler--Poincar\'e characteristic for the sheaf $\mathcal{M}$ by
 \[	\chi_{2}(U,\mathcal{M}):=\dfrac{[H^{0}_{\et}(U,\mathcal{M})][H_{\et}^{2}(U,\mathcal{M})]}{[H_{\et}^{1}(U,\mathcal{M})]}. \]
	 
	\begin{Thm}\label{calculationofetale}
	With the notation as above, if $M$ is $p$-primary, then we have
	\begin{align*}
	\chi_{2}(U,\mathcal{M})=& \dfrac{1}{[M]^{[K:\mathbb{Q}]}}\prod_{v\in S_{\infty}}[H^{0}(G_{v},M)]\cdot
	\prod_{v\notin S, v\in S_{p}} \dfrac{1}{|[M]|_{v}}\cdot \epsilon, 
	\end{align*}
where
	\[ 
\epsilon := \begin{cases} 
	1, & \text{if } S_{\mathrm{fin}} \neq \emptyset, \\ 
	[(M')^{G_{K}}], & \text{if } S_{\mathrm{fin}} = \emptyset, \text{ with } p\neq 2 \text{ or } K \text{ totally imaginary}, \\ 
	\left[ \bigcap_{v\in S_{\mathbb{R}}} \left( (1+c_{v})M' \cap (M')^{G_{K}} \right) \right], & \text{if } S_{\mathrm{fin}} = \emptyset, \text{ with } p=2 \text{ and } K \text{ not totally imaginary}.
\end{cases} 
\]
Here, $c_{v} \in G_{K}$ denotes a chosen complex conjugation for each $v \in S_{\mathbb{R}}$.
	\end{Thm}
	
	\begin{proof}
	Since $M$ is $p$-primary, we have 
		\[ \prod_{v\notin S \cup S_{\infty}} 1/|[M]|_v=\prod_{v\notin S,v\in S_{p}}1/|[M]|_v.\]
		By Theorem \ref{eulercharacteristicformulaetale}, it remains to calculate
		\begin{equation}\label{55}
		\epsilon:=\prod_{v\in S_{\infty}} \dfrac{1}{[\widehat{H}^{0}(G_{v},M)]} \cdot [H^{3}_{\et}(U,\mathcal{M})].
		\end{equation}
	
	By (\ref{qiechu}) and Proposition \ref{gaoweitrivial}, we have an exact sequence of abelian groups
	\begin{equation}\label{xulie}
		\prod_{v\in S} H_{\et}^{2}(K_{v},\mathcal{M})\xrightarrow{\alpha}H^{3}_{c}(U,\mathcal{M}) \to H^{3}_{\et}(U,\mathcal{M})\to \prod_{v\in S} H_{\et}^{3}(K_{v},\mathcal{M})\to 0.
	\end{equation}
	By \cite[Theorem 7.1.8 and Proposition 1.7.6]{MR2392026}, we have
	\begin{equation}\label{guji}
	\prod_{v\in S}[H^{3}_{\et}(K_{v},\mathcal{M})]= \prod_{v\in S_{\infty}}[\widehat{H}^{0}(G_{v},M)].
	\end{equation}
	Combining these with (\ref{55}) yields
   \[ \epsilon=[\operatorname{coker}(\alpha)]. \]

    By Lemma \ref{calculationofH3}, we have
	\begin{equation}\label{artinverdierdua}
   	H^{3}_{c}(U,\mathcal{M})^{\vee}\cong (M')^{G_{K}}.
	\end{equation}

Let $S_{\mathbb{R}}$ denote the subset of $S_{\infty}$ consisting of real primes. By the local duality theorem \cite[Theorem 7.2.6 and Theorem 7.2.17]{MR2392026}, we have 
	 \begin{align*}
	 	\left( \prod_{v\in S_{\mathrm{fin}}} H_{\et}^{2}(K_{v},\mathcal{M}) \right) ^{\vee}\cong &  \bigoplus_{v\in S_{\mathrm{fin}}}H^{0}(G_{v},M'),
	 \end{align*}
    and 
    \[ \left( \prod_{v\in S_{\infty}} H_{\et}^{2}(K_{v},\mathcal{M})\right) ^{\vee}= \left( \prod_{v\in S_{\mathbb{R}}} H_{\et}^{2}(K_{v},\mathcal{M})\right)^{\vee}\cong  \bigoplus_{v\in S_{\mathbb{R}}}\widehat{H}^{0}(G_{v},M') .\]
    
	If $S_{\mathrm{fin}}\neq \emptyset$, then the dual of the restriction of $\alpha$ to $\prod_{v\in S_{\mathrm{fin}}}H^{2}_{\et}(K_{v},\mathcal{M})$ is the natural inclusion
	 \[    	H^{3}_{c}(U,\mathcal{M})^{\vee}\cong (M')^{G_{K}}\hookrightarrow \bigoplus_{v\in S_{\mathrm{fin}}}(M')^{G_{v}},\]
	 which is injective. It follows that $\alpha$ is surjective, whence $\epsilon=1$.
	 
	 Next, suppose that $S_{\text{fin}} = \emptyset$. Assume further that either $p \neq 2$ or $K$ is totally imaginary. Then we have $\prod_{v\in S}H_{\et}^{2}(K_{v},\mathcal{M})=0$, and hence $\epsilon=[(M')^{G_{K}}]$ by (\ref{artinverdierdua}).
	 
	 If $p=2$, $K$ is not totally imaginary and $S_{\mathrm{fin}}=\emptyset$, then the dual map $\alpha^{\vee}$ of $\alpha$ is the morphism
	 \[ (M')^{G_{K}}\to \bigoplus_{v\in S_{\mathbb{R}}} (M')^{c_{v}}/(1+c_{v})M', \]
	 where $c_{v} \in G_{K}$ denotes a chosen complex conjugation for each $v\in S_{\mathbb{R}}$. Thus, we obtain
	 \[ \epsilon=[\operatorname{coker}(\alpha)]=[\ker(\alpha^{\vee})]=\left[ \bigcap_{v\in S_{\mathbb{R}}}\left( (1+c_{v})M' \cap (M')^{G_{K}} \right)\right]. \]
	 This completes the proof of our theorem.
	\end{proof}
	
	\section{Proof of Theorem \ref{main}}\label{proofofmaintheorem}

\subsection{First observations}	Let $K$ be a number field and $S$ a finite set of primes of $K$ containing all archimedean primes. In this section, we will use the following notation.
	\begin{Convention}
		\begin{itemize}
			\item For a field extension $L$ of $K$, by abuse of notation, we will write $S=S(L)$ for the primes of $L$ above the primes of $K$ in $S$, and $\mathcal{O}_{L,S}$ for the integral closure of $\mathcal{O}_{K,S}$ in $L$. 
			\item If $L/K$ is a finite extension and $T$ is a set of primes of $L$, then we define
			\[ T_{K}:=\{\text{$v$ is a prime of $K$} \mid \text{there exists a prime $w\in T $ above $v$} \}.\]
		\end{itemize}
	\end{Convention}
	
	 Let $M$ be a finite discrete $G_{K,S}$-module, and let $\mathcal{M}$ be the corresponding locally constant sheaf on $\operatorname{Spec}(\mathcal{O}_{K,S})$. Since the morphism $\operatorname{Spec}(\mathcal{O}_{K_{S},S})\to \operatorname{Spec}(\mathcal{O}_{K,S})$ is a Galois covering with Galois group $G_{K,S}$, by \cite[Theorem 2.20 and Remark 2.21(b), Chapter III]{MR559531}, we have the Hochschild--Serre spectral sequence 
	\begin{equation}\label{Hochserre}
			 E_{2}^{r,s}=H^{r}(G_{K,S},H^{s}_{\et}(\operatorname{Spec}(\mathcal{O}_{K_{S},S}),\mathcal{M}))\Rightarrow H_{\et}^{r+s}(\operatorname{Spec}(\mathcal{O}_{K,S}),\mathcal{M}).
	\end{equation}
	
	The edge homomorphisms are the morphisms
	\begin{equation}\label{qiaoliang}
		\phi_{i}:H^{i}(G_{K,S},M)\to H_{\et}^{i}(\operatorname{Spec}(\mathcal{O}_{K,S}),\mathcal{M}).
	\end{equation}
	Note that
	\[
	H^{1}_{\et}(\operatorname{Spec}(\mathcal{O}_{K_{S},S}),\mathcal{M}) = \operatorname{Hom}(\pi_1(\operatorname{Spec}(\mathcal{O}_{K_{S},S}),\bar{\eta}),M) = 0
	\]
   because $\pi_1(\operatorname{Spec}(\mathcal{O}_{K_{S},S}),\bar{\eta})=0$. 
	It follows that the morphism $\phi_{i}$ is an isomorphism for $i\leq 1$, and is injective for $i=2$. 
	
	Similarly, if $M$ is a finite discrete $p$-primary $G_{K,S}(p)$-module, then we have homomorphisms 
		\begin{equation}
		\phi_{i}(p):H^{i}(G_{K,S}(p),M)\to H_{\et}^{i}(\operatorname{Spec}(\mathcal{O}_{K,S}),\mathcal{M})
	\end{equation}
  which factor as follows:
	 \[   \xymatrix{
	 	H^{i}(G_{K,S}(p),M)\ar[d]_{\operatorname{Inf}^{i}} \ar[r]^{\phi_{i}(p)}& H^{i}_{\et}(\operatorname{Spec}(\mathcal{O}_{K,S}),\mathcal{M})\\
	 	H^{i}(G_{K,S},M)\ar[ur]_{\phi_{i}}&
	 }\]
	where $\operatorname{Inf}^{i}:H^{i}(G_{K,S}(p),M)\to H^{i}(G_{K,S},M)$ is the inflation map. Since $M$ is $p$-primary, we have
	\[ H^{1}(\operatorname{Gal}(K_{S}/K_{S}(p)),M)=0, \]
	and hence the map
	\begin{equation}\label{inf2isinjective}
		\operatorname{Inf}^{2}:H^{2}(G_{K,S}(p),M)\to H^{2}(G_{K,S},M)
	\end{equation}
	 is injective. Since $\phi_2$ is injective, it follows that if $\phi_2(p)$ is an isomorphism, then so is $\phi_2$.
	
	For any prime $v$ of $K$, the composite morphism $\mathcal{O}_{K,S}\to K\to K_{v}$ induces natural morphisms $H^{i}_{\et}(\operatorname{Spec}(\mathcal{O}_{K,S}),\mathcal{M})\to H^{i}_{\et}(K_{v},\mathcal{M})$ for all $i\geq 0$. We define
	\[ \Sha^{i}(K,S,\mathcal{M}):=\ker \left( H^{i}_{\et}(\operatorname{Spec}(\mathcal{O}_{K,S}),\mathcal{M})\to \prod_{v\in S}H^{i}_{\et}(K_{v},\mathcal{M})\right),\]
	for any $i\geq 0$.
	
	\subsection{Removing the $h^{2}$-defect in the pro-$p$ extension case}
	In this subsection, we fix a prime number $p$. Let $K$ be a number field and $S$ a finite set of primes of $K$. For the remainder of this subsection, we assume that $p \neq 2$ or that $K$ is totally imaginary. Assume also that $ \mathbb{Z}/p\mathbb{Z}$ is equipped with the trivial action of $G_{K,S}$. We write $\Sha_{p}^{i}(K,S) := \Sha^{i}(K,S,\mathbb{Z}/p\mathbb{Z})$. We also define
  \[ 
  V_{S,p}(K) := \big\{ \alpha \in K^{\times} \mid (\alpha) = I^p \text{ for some fractional ideal } I, \text{ and } \alpha \in K_{v}^{\times p} \text{ for all } v \in S \big\} \big/ K^{\times p},
  \]
  and we set $\Bcyr_{S,p}(K) := V_{S,p}(K)^{\vee}$. It then follows from \cite[Theorem 3.6]{MR2365909} that there is a natural isomorphism
  \[ 
  \Sha_{p}^{2}(K,S) \simeq \Bcyr_{S,p}(K). 
  \]
For simplicity, throughout this subsection we will write $H^{i}(G_{K,S}(p))$ for $H^{i}(G_{K,S}(p),\mathbb{Z}/p\mathbb{Z})$ and $H^{i}_{\et}(\operatorname{Spec}(\mathcal{O}_{K,S}))$ for $H^{i}_{\et}(\operatorname{Spec}(\mathcal{O}_{K,S}),\mathbb{Z}/p\mathbb{Z})$ for any integer $i\geq 0$.

Following \cite[Section 4]{MR2365909}, we introduce the following terminology.

\begin{Def}
	Keeping the above notation, we make the following definitions:
	\begin{enumerate}
		\item We define the $h^{2}$-defect of $S$ with respect to $p$ as
		\[ \delta^{2}_{S,p}(K) := \dim_{\mathbb{F}_{p}}H^{2}_{\et}(\operatorname{Spec}(\mathcal{O}_{K,S})) - \dim_{\mathbb{F}_{p}}H^{2}(G_{K,S}(p)). \]
		\item We denote by $K_{S}^{\mathrm{el}}$ the maximal elementary abelian $p$-extension of $K$ contained in $K_{S}(p)$.
		\item We say that a nonarchimedean prime $v$ of $K$ is tame if $N(v) \equiv 1 \pmod p$, where $N(v)$ denotes the norm of $v$, i.e., the cardinality of the residue field of $v$.
		\item Let $S$ be a set of tame primes of $K$. Let $\zeta_p$ denote a primitive $p$-th root of unity in $\overline{K}$. We define the $p$-density $\Delta_{K}^{p}(S)$ as
		\[
		\Delta_{K}^{p}(S) =\delta_{K(\zeta_p)}(S(K(\zeta_p)))
		\]
		where $S(K(\zeta_p)) $ is the set of primes of $K(\zeta_p)$ lying above $S$ and $\delta_{K(\zeta_p)}$ denotes the Dirichlet density on the level $K(\zeta_p)$.
	\end{enumerate}
\end{Def}

   \begin{Lem}\cite[Lemma 4.4]{MR2365909}\label{descenddefect}
	Let $K$ be a number field and $S$ a finite set of primes of $K$ such that $\Bcyr_{S,p}(K)=0$. Let $v$ be a tame prime of $K$ that does not split completely in $K_{S}^{\mathrm{el}}$. Then
	\[ \delta^{2}_{S\cup \{v\},p}(K)\leq \delta_{S,p}^{2}(K). \]
	Moreover, the natural morphism $H^{2}(G_{K,S\cup \{v\}}(p))\to H^{2}(G_{v})$ is surjective, where $H^{2}(G_{v}):=H^{2}(G_{v},\mathbb{Z}/p\mathbb{Z})$.
\end{Lem}

\begin{Lem}\label{largeentough}
	Let $K$ be a number field, and let $S$ be a set of tame primes of $K$ with $p$-density $\Delta^{p}_{K}(S)=1$. If $S_{1}$ is any finite set of primes of $K$, then the natural morphisms
	\[ 	\phi_{i}(p):H^{i}(G_{K,S\cup S_{1}}(p))\to H_{\et}^{i}(\operatorname{Spec}(\mathcal{O}_{K,S\cup S_{1}}))\]
	are isomorphisms for all $i \ge 0$.
\end{Lem}
 \begin{proof}
 	The assertion follows by applying the same argument as in the proof of \cite[Theorem 4.3]{MR2365909}. 
 \end{proof}
 
  Let $K'/K$ be a finite Galois extension of number fields and $S'$ be a finite set of primes of $K'$. Note that $S'$ is $\operatorname{Gal}(K'/K)$-stable if and only if $ S'$ contains all primes of $K'$ that lie above the primes in $(S')_K$ where $(S')_{K}$ denotes the primes of $K$ below $S'$.

    \begin{Lem}\label{schmidt}
  	Let $K$ be a number field, $S$ a finite set of primes of $K$ and $\mathcal{T}$ a set of primes of $K$ of Dirichlet density zero such that $\mathcal{T}\cap S=\emptyset$. Let $K'/K$ be a finite Galois extension inside $K_{S}$. Let $S_{1}$ be the set of all tame primes of $K$ not in $S\cup \mathcal{T}$, and let $S'_{1}$ be the set of primes of $K'$ lying above $S_{1}$. Let $S'$ be the set of primes of $K'$ lying above $S$. Then the following assertions hold:
  	\begin{enumerate}
  		\item The natural morphisms
  		\[ 	\phi_{i}(p):H^{i}(G_{K',S'\cup S'_{1}}(p))\to H_{\et}^{i}(\operatorname{Spec}(\mathcal{O}_{K',S'\cup S'_{1}}))\]
  		are isomorphisms for all $i$.
  		
  		\item There exists a finite subset $S'_{0}$ of $S'_{1}$ such that $S'_{0}$ is $\operatorname{Gal}(K'/K)$-stable and $\Bcyr_{S'\cup S'_{0},p}(K') = 0$.
  	\end{enumerate}
  \end{Lem}
  \begin{proof}
  	Note that the set of all tame primes of $K$ has $p$-density equal to $1$. Since $S$ is finite and $\mathcal T$ has Dirichlet density zero, we have $	\delta_{K(\zeta_p)}\bigl(S_1(K(\zeta_p))\bigr)=1$. Since $K'(\zeta_p)/K(\zeta_p)$ is a finite extension, we have
  	\[
  	\delta_{K'(\zeta_p)}\bigl(S_1(K'(\zeta_p))\bigr)=1.
  	\]
  	But $S_1(K'(\zeta_p))$ is precisely the set of primes of $K'(\zeta_p)$ lying above
  	primes of $S'_1$. Therefore, by the definition of $p$-density,
  	\[
  	\Delta^p_{K'}(S'_1)
  	=
  	\delta_{K'(\zeta_p)}\bigl(S'_1(K'(\zeta_p))\bigr)
  	=
  	\delta_{K'(\zeta_p)}\bigl(S_1(K'(\zeta_p))\bigr)
  	=
  	1.
  	\]
   Then the first assertion follows from Lemma \ref{largeentough} and the second assertion follows from \cite[Proposition 4.2]{MR2365909}. 
  \end{proof}

  We adapt the proof of \cite[Lemma 4.5]{MR2365909} to establish the following result.
  	
		\begin{Lem}\label{technicallemma}
		Let $K'/K$ be a finite Galois extension of number fields. Let $S'$ be a finite set of primes of $K'$ such that 
		\begin{itemize}
			\item $\Bcyr_{S',p}(K') = 0$,
			\item $K'^{\mathrm{el}}_{S'} \neq K'$,
			\item $ S'$ is $\operatorname{Gal}(K'/K)$-stable.
		\end{itemize}
	   Let $\mathcal{T}$ be a set of primes of $K'$ of Dirichlet density zero such that $\mathcal{T} \cap S' = \emptyset$. Let $\mathfrak{p}$ be a tame prime of $K'$ such that $\mathfrak{p}\notin S'\cup \mathcal{T}$ and $\mathfrak{p}$ splits completely in $K'^{\mathrm{el}}_{S'}/K'$. Then there exists a tame prime
		$\mathfrak p_0$ of $K$ such that, if $S'_{\mathfrak p_0}$ denotes the set of
		primes of $K'$ lying above $\mathfrak p_0$, then
		\begin{enumerate}
			\item $S'_{\mathfrak p_0}\cap (S'\cup \mathcal T)=\emptyset$;
			\item every prime $\mathfrak p'\in S'_{\mathfrak p_0}$ does not split
			completely in $K'^{\mathrm{el}}_{S'}/K'$;
			\item $\mathfrak p$ does not split completely in
			$K'^{\mathrm{el}}_{S'\cup S'_{\mathfrak p_0}}/K'$.
		\end{enumerate}
		In particular,
		\[
		\delta^{2}_{S'\cup \{\mathfrak p\}\cup S'_{\mathfrak p_0},p}(K')
		\leq
		\delta^{2}_{S',p}(K').
		\]
	\end{Lem}
	
	\begin{proof}
		Since $\mathfrak{p}$ splits completely in $K'^{\mathrm{el}}_{S'}/K'$, by class field theory, there exists an $s\in K'^{\times}$ such that
		\begin{enumerate}[label=(\alph*)]\label{liejutiaojian}
			\item $v_{\mathfrak{p}}(s) \equiv 1 \pmod p$,
			\item $v_{\mathfrak{q}}(s) \equiv 0 \pmod p$ for all $\mathfrak{q} \notin S' \cup \{\mathfrak{p}\}$,
			\item $s \in K'^{\times p}_{\mathfrak{q}}$ for all $\mathfrak{q} \in S'$.
		\end{enumerate}
	Since $\Bcyr_{S',p}(K')=0$, $s$ is well-defined as an element in $K'^{\times}/K'^{\times p}$.
		
	Since $v_{\mathfrak{p}}(s) \not\equiv 0 \pmod p$, $s$ is not a global $p$-th power in $K'$, and the extension $K'(\zeta_p, \sqrt[p]{s})/K'(\zeta_p)$ is cyclic of degree $p$. Consider the following two non-trivial extensions over $K'(\zeta_p)$:
		\[
		M_1 = K'^{\mathrm{el}}_{S'}(\zeta_p) \quad \text{and} \quad M_2 = K'(\zeta_p, \sqrt[p]{s}).
		\]
		Since $\mathfrak p$ is tame, we have $\zeta_p\in K'_{\mathfrak p}$. For every prime $\mathfrak P$ of $K'(\zeta_p)$ lying above
		$\mathfrak p$, one has $v_{\mathfrak P}(s)=v_{\mathfrak p}(s)\equiv 1 \pmod p$. By local Kummer theory, the local extension $	K'(\zeta_p)_{\mathfrak P}(\sqrt[p]{s})/
		K'(\zeta_p)_{\mathfrak P}$ is ramified. Hence $M_2/K'(\zeta_p)$ is ramified at $\mathfrak{P}$. Since $\mathfrak{p} \notin S'$, the extension $M_1/K'(\zeta_p)$ is unramified at $\mathfrak{P}$. It follows that $M_1$ and $M_2$ are linearly disjoint over $K'(\zeta_p)$. Thus, their compositum $M_1 M_2$ is an abelian extension over $K'(\zeta_p)$ with Galois group $\operatorname{Gal}(M_1 M_2/K'(\zeta_p)) \cong \operatorname{Gal}(M_1/K'(\zeta_p)) \times \operatorname{Gal}(M_2/K'(\zeta_p))$.
	
\begin{tikzcd}[row sep=1.5em, column sep=1.5em]
	& & K'^{\mathrm{el}}_{S'} \arrow[drr, dash] & & & & \\
	& & & & M_1 = K'^{\mathrm{el}}_{S'}(\zeta_p) \arrow[dr, dash] & & \\
	K \arrow[r, dash] & K' \arrow[uur, dash] \arrow[rr, dash] & & K'(\zeta_p) \arrow[ur, dash] \arrow[dr, dash] & & M_1M_2 \arrow[r, dash] & N \\
	& & & & M_2 = K'(\zeta_p, \sqrt[p]{s}) \arrow[ur, dash] & &
\end{tikzcd}
		
	 Let $N$ be the Galois closure of $M_1 M_2$ over $K$. Since $M_1/K'(\zeta_p)$ and $M_2/K'(\zeta_p)$ are non-trivial, we can choose an element $\sigma_0 \in \operatorname{Gal}(N/K'(\zeta_{p}))$ such that
		\[
		\sigma_0\big|_{M_1} \neq \text{id} \quad \text{and} \quad \sigma_0\big|_{M_2} \neq \text{id}.
		\]
		Let $C \subset \operatorname{Gal}(N/K)$ be the conjugacy class of $\sigma_0$. Since the extension $K'(\zeta_p)/K$ is Galois, we have $C\subset \operatorname{Gal}(N/K'(\zeta_{p}))$. Since $S'$ is $\operatorname{Gal}(K'/K)$-stable, the extension $K'^{\mathrm{el}}_{S'}/K$ is Galois, and so is the extension $M_{1}/K$. Thus, every $\sigma \in C$ fixes $K'(\zeta_p)$ and acts non-trivially on $M_1$. By the Chebotarev density theorem applied to the Galois extension $N/K$, there exist infinitely many primes $\mathfrak{p}_0$ of $K$ (unramified in $N$) with positive Dirichlet density whose associated Frobenius conjugacy class is precisely $C$, i.e., $\text{Frob}_{N/K}(\mathfrak{p}_0) = C$. We choose one such $\mathfrak{p}_0$ outside the set $(S'\cup \mathcal{T})_{K}$, which is permissible since $\mathcal{T}$ has Dirichlet density zero and $S'$ is finite. Since elements in $C$ fix $\zeta_p$, the prime $\mathfrak{p}_0$ splits completely in $K(\zeta_p)/K$, forcing $N(\mathfrak{p}_0) \equiv 1 \pmod p$.
		
		Now, let $\mathfrak{p}'$ be any prime of $K'$ lying over $\mathfrak{p}_0$. The Frobenius conjugacy class $\text{Frob}_{N/K'}(\mathfrak{p}')$ is contained in $C$. By our construction, every element in $C$ restricts non-trivially to $M_1 = K'^{\text{el}}_{S'}(\zeta_p)$. Thus, $\text{Frob}_{K'^{\mathrm{el}}_{S'}/K'}(\mathfrak{p}') \neq \text{id}$ and hence $\mathfrak{p}'$ does not split completely in $K'^{\text{el}}_{S'}/K'$.
		
		Moreover, by choosing a prime $\mathfrak{P}$ of $N$ above $\mathfrak p_0$ whose Frobenius
		element is $\sigma_0$, we obtain a prime $\mathfrak{p}'_{0}=\mathfrak{P}\cap K'\in S'_{\mathfrak p_0}$. Let $\mathfrak{P}_0 := \mathfrak{P} \cap K'(\zeta_p)$. Since the Frobenius element of $\mathfrak{P}$ in $\operatorname{Gal}(N/K'(\zeta_p))$ is exactly $\sigma_0$, and $\sigma_0$ acts non-trivially on $M_2$, the prime $\mathfrak{P}_0$ does not split completely in $M_2/K'(\zeta_p)$. By local Kummer theory, this is equivalent to $s\notin K'(\zeta_{p})_{\mathfrak P_0}^{\times p}=K'^{\times p}_{\mathfrak p'_0}$.
		
	  Now suppose for contradiction that $\mathfrak{p}$ splits completely in $K'^{\mathrm{el}}_{S'\cup \{\mathfrak{p}'_{0}\}}/K'$. Then there exists an element $t \in K'^{\times}$ satisfying conditions (a)–(c) with respect to the enlarged set $S' \cup \{\mathfrak{p}'_0\}$, which in particular implies $t \in K'^{\times p}_{\mathfrak{p}'_{0}}$. Since $s/t\in \Bcyr_{S',p}(K')=0$, we have $s/t\in K'^{\times p}$ which implies that $s\in K'^{\times p}_{\mathfrak{p}'_{0}}$. This contradicts our earlier deduction that $s \notin K'^{\times p}_{\mathfrak{p}'_{0}}$. Thus, $\mathfrak{p}$ cannot split completely in $K'^{\mathrm{el}}_{S' \cup \{\mathfrak{p}'_{0}\}}/K'$. It immediately follows that $\mathfrak{p}$ does not split completely in $K'^{\mathrm{el}}_{S' \cup S'_{\mathfrak{p}_0}}/K'$.
	  
	  Finally, by applying Lemma \ref{descenddefect}, we conclude that $\delta^{2}_{S' \cup \{\mathfrak{p}\} \cup S'_{\mathfrak{p}_{0}},p}(K') \leq \delta^{2}_{S',p}(K')$.
	\end{proof}
	
	Furthermore, we adapt the proof of \cite[Lemma 4.6]{MR2365909} to establish the following result.
	
			\begin{Thm}\label{control}
		Let $K'/K$ be a finite Galois extension of number fields. Let $S'$ be a $\operatorname{Gal}(K'/K)$-stable finite set of primes of $K'$ and $\mathcal{T}$ be a set of primes of $K'$ of Dirichlet density zero, disjoint from $S'$. Suppose that $M$ is a finite discrete $p$-primary $G_{K',S'}(p)$-module with trivial action. Then there exists a finite set $S_{0}$ of primes of $K'$ satisfying the following:
		\begin{enumerate}
			\item $S_{0} \cap (S'\cup \mathcal{T})= \emptyset$;
			\item $(S_{0}\cup S')_{\mathrm{fin}}\neq \emptyset$;
			\item if $M$ is viewed as a $G_{K',S'\cup S_{0}}$-module, then the natural morphism
			\[ \phi_{2}:H^{2}(G_{K',S'\cup S_{0}},M)\to H^{2}_{\et}(\operatorname{Spec}(\mathcal{O}_{K',S'\cup S_{0}}),\mathcal{M}) \]
			is an isomorphism;
			\item $S_0$ is $\operatorname{Gal}(K'/K)$-stable.
		\end{enumerate}
	\end{Thm}
	\begin{proof}
	Since $M$ is $p$-primary and the $G_{K',S'}(p)$-action on it is trivial, we may assume without loss of generality that $M = \mathbb{Z}/p\mathbb{Z}$. Condition (3) then reduces to showing that the natural morphism
			\[ \phi_{2}(p):H^{2}(G_{K',S'\cup S_{0}}(p))\to H^{2}_{\et}(\operatorname{Spec}(\mathcal{O}_{K',S'\cup S_{0}})) \]
			is an isomorphism.
			
    Let $S_{1}$ denote the set of all tame primes of $K$ not in $(S'\cup \mathcal{T})_{K}$, and let $S'_{1}$ denote the set of primes of $K'$ lying above $S_{1}$. By Lemma \ref{schmidt}, the morphism
    \[ 	\phi_{2}(p):H^{2}(G_{K',S'\cup S'_{1}}(p))\to H_{\et}^{2}(\operatorname{Spec}(\mathcal{O}_{K',S'\cup S'_{1}}))\]
    is an isomorphism, and there exists a $\operatorname{Gal}(K'/K)$-stable finite subset $S'_{0}$ of $S'_{1}$  such that $\Bcyr_{S'\cup S'_{0},p}(K') = 0$. 
    
    After replacing $S'$ by $S'\cup S'_{0}$, we may assume that $\Bcyr_{S',p}(K') = 0$. Then $\Sha^{2}_{p}(K',S')=0$. Thus, for any $T\supset S'$, the natural morphisms $H^{2}(G_{K',S'}(p))\to H^{2}(G_{K',T}(p)) $ and $H^{2}_{\et}(\operatorname{Spec}(\mathcal{O}_{K',S'}))\to H^{2}_{\et}(\operatorname{Spec}(\mathcal{O}_{K',T}))$ are injective. If $\delta^{2}_{S',p}(K')=0$, then we are done. Otherwise, we consider the following commutative diagram of $\mathbb{F}_{p}$-vector spaces
	\[
	\begin{tikzcd}[column sep=large, row sep=large]
		H^2(G_{K',S'}(p)) \arrow[r, hook] \arrow[d, hook] & 
		H^2(G_{K',S'\cup S'_1}(p)) \arrow[d, "\cong"] \\
		H^2_{\text{ét}}(\text{Spec}(\mathcal{O}_{K',S'})) \arrow[r, hook] & 
		H^2_{\text{ét}}(\text{Spec}(\mathcal{O}_{K',S'\cup S'_1})).
	\end{tikzcd}
	\]
	
	 Let $x\in	H^{2}_{\et}(\operatorname{Spec}(\mathcal{O}_{K',S'})) $ but $x\notin 	H^{2}(G_{K',S'}(p))$. Since $S'_{1}$ is $\operatorname{Gal}(K'/K)$-stable, there exists a finite subset $S'_{2}$ of $S'_{1}$ such that $x\in H^{2}(G_{K',S'\cup S'_{2}}(p))$ and $S'_{2}$ is $\operatorname{Gal}(K'/K)$-stable. Let $S'_{3}=\{\mathfrak{p}_{1},\dots,\mathfrak{p}_{n}\} \subset S'_{2}$ be the subset consisting of the primes that split completely in $K'^{\mathrm{el}}_{S'}$. We choose $\mathfrak{p}'_{1},\cdots,\mathfrak{p}'_{n}\in (S'_{1})_{K}$ according to Lemma \ref{technicallemma} and let $S'_{4}$ be the finite set of primes of $K'$ above the set $\{\mathfrak{p}'_{1},\cdots,\mathfrak{p}'_{n}\}$. By Lemma \ref{descenddefect}, the natural morphism
	 \[ \varphi_{1}:H^{2}(G_{K',S'\cup S'_{2}\cup S'_{4}}(p))\to \prod_{\mathfrak{p}\in S'_{2}\cup S'_{4}} H^{2}(G_{K'_{\mathfrak{p}}}) \]
	 is surjective.
	 
	Now, we consider the following commutative diagram of $\mathbb{F}_{p}$-vector spaces with exact rows:
\[
\begin{tikzcd}[row sep=large, column sep=small]
	0 \arrow[r] & H^2(G_{K',S'}(p)) \arrow[d, "\alpha", hook] \arrow[r, "f"] & 
	H^2(G_{K',S' \cup S'_2 \cup S'_4}(p)) \arrow[d, "\beta", hook] \arrow[dr, "\varphi_1"', two heads] \arrow[r] & 
	\operatorname{coker} f \arrow[d, "\varphi_2", two heads] \arrow[r] & 0 \\
	0 \arrow[r] & H^2_{\et}(\operatorname{Spec}(\mathcal{O}_{K', S'})) \arrow[r] & 
	H^2_{\et}(\operatorname{Spec}(\mathcal{O}_{K', S' \cup S'_2 \cup S'_4})) \arrow[r] & 
	\displaystyle\prod_{\mathfrak{p} \in S'_2 \cup S'_4} H^2(G_{K'_{\mathfrak{p}}}) \arrow[r] & 0.
\end{tikzcd}
\]
By the snake lemma, we have an exact sequence of $\mathbb{F}_{p}$-vector spaces
\[ \ker(\varphi_{2})\to \operatorname{coker}(\alpha)\to \operatorname{coker}(\beta)\to \operatorname{coker}(\varphi_{2}), \]
where $\operatorname{coker}(\varphi_{2})=0$. Since $x\in \ker(\varphi_{1})$ and $x\notin H^{2}(G_{K',S'}(p))$, we deduce that $\ker(\varphi_{2})\neq 0$. It follows that $\delta^{2}_{S'\cup S'_{2}\cup S'_{4},p}(K')<\delta^{2}_{S',p}(K')$. By iterating this process, we eventually obtain a finite set $S_{0}$ of primes that has a trivial $h^{2}$-defect and satisfies the required conditions.
\end{proof}

   \subsection{A key theorem}
   In this subsection, we aim to prove the following key theorem, which plays a crucial role in the proof of the third part of Theorem \ref{main}.

	\begin{Thm}\label{main2}
		Let $K$ be a number field and let $p$ be a prime number. Assume that $p\neq 2$ or that $K$ is totally imaginary. Let $S$ be a finite set of primes of $K$ and $\mathcal{T}$ a set of primes of Dirichlet density zero such that $\mathcal{T}\cap S=\emptyset$. 
		Suppose that $M$ is a finite discrete $p$-primary $G_{K,S}$-module.
		Then there exists a finite set $S_{0}$ of primes of $K$ satisfying the following:
		\begin{enumerate}
			\item $S_{0}$ is disjoint from $S\cup \mathcal{T}$;
			\item $(S_{0}\cup S)_{\mathrm{fin}}\neq \emptyset$;
			\item if $M$ is viewed as a $G_{K,S\cup S_{0}}$-module, then the natural morphisms
			\[  \phi_{i}:H^{i}(G_{K,S\cup S_{0}},M)\to H^{i}_{\et}(\operatorname{Spec}(\mathcal{O}_{K,S\cup S_{0}}),\mathcal{M}) \]
			are isomorphisms for all $i\leq 2$. 
		\end{enumerate}
	\end{Thm}
	
	For this purpose, we need the following algebraic lemma.
	
		\begin{Lem}\label{homologicalfact}
		Let $F: E^{r,s}_{m} \to \mathcal{E}_{m}^{r,s}$ be a morphism of first quadrant cohomological spectral sequences in the category of abelian groups. Suppose that $F$ is an isomorphism for all bigrades $(r,s) $ with $s \leq 2$ at the $E_2$ page. Then $F$ induces an isomorphism on $H^2$.
	\end{Lem}
	\begin{proof}
		Suppose we have two first quadrant cohomological spectral sequences, $E_m^{r,s}$ and $\mathcal{E}_m^{r,s}$, converging to filtered graded objects $H^*$ and $\mathcal{H}^*$ respectively. We are given a morphism of spectral sequences $F_{m}: E^{r,s}_{m} \to \mathcal{E}_{m}^{r,s}$ such that the map on the second page, $F_2: E_2^{r,s} \to \mathcal{E}_2^{r,s}$, is an isomorphism for all $r \ge 0$ and for $s \in \{0, 1, 2\}$. The group $H^2$ has a finite filtration:
		\[ 0 = F^{3}H^{2} \subseteq F^{2}H^{2} \subseteq F^{1}H^{2} \subseteq F^{0}H^{2} = H^2. \]
		The associated graded pieces are the terms on the $E_\infty$-page that lie on the diagonal $r+s=2$:
		\[ E_\infty^{2,0} \cong F^{2}H^{2}/F^{3}H^{2},
		E_\infty^{1,1} \cong F^{1}H^{2}/F^{2}H^{2},
		E_\infty^{0,2} \cong F^{0}H^{2}/F^{1}H^{2}. \]
		To show that the induced map $F_*: H^2 \to \mathcal{H}^2$ is an isomorphism, by the five lemma, it suffices to show that $F$ induces an isomorphism on each of these associated graded pieces, i.e., that $F: E_\infty^{r,s} \to \mathcal{E}_\infty^{r,s}$ is an isomorphism for $(r,s) = (2,0), (1,1),$ and $(0,2)$.
		\begin{itemize}
			\item  The $ (2,0)$ term. The term $E_\infty^{2,0}$ is computed from the $E_2$ page. The only relevant differential is $d_2: E_2^{0,1} \to E_2^{2,0}$, and we have
			$E_3^{2,0} =\operatorname{coker}(d_2: E_2^{0,1} \to E_2^{2,0})$.
			Then we have the following commutative diagram.
			\[ \xymatrix{
				E_2^{0,1} \ar_{F_{2}}[d] \ar[r]^{d_{2}} &E_2^{2,0}\ar[d]_{F_{2}}\\
				\mathcal{E}_2^{0,1}  \ar[r]^{d'_{2}} &\mathcal{E}_2^{2,0}} \]
			By assumption, the vertical maps are isomorphisms which implies that the induced map on the cokernels, $F_{3}: E_3^{2,0} \to \mathcal{E}_3^{2,0}$, is also an isomorphism.
			Since the later differentials are zero, $E_\infty^{2,0} = E_3^{2,0}$. Thus, $F_{\infty}: E_\infty^{2,0} \to \mathcal{E}_\infty^{2,0}$ is an isomorphism.
			\item The $(1,1)$ term. In this case, the only relevant differential is $d_2: E_2^{1,1} \to E_2^{3,0}$, and we have $E_3^{1,1} = \ker(d_2: E_2^{1,1} \to E_2^{3,0})$. Since $F_{2}$ commutes with all differentials and $F_2$ is an isomorphism on the source and the target, the induced map on the kernels, $F_3: E_3^{1,1} \to \mathcal{E}_3^{1,1}$, is an isomorphism. Since later differentials are zero, $E_\infty^{1,1} = E_3^{1,1}$. Thus, $F_{\infty}: E_\infty^{1,1} \to \mathcal{E}_\infty^{1,1}$ is an isomorphism.
			\item The $ (0,2)$ term.
			This term can be affected by two differentials, $d_2$ and $d_3$.
			First, we have $E_3^{0,2} = \ker(d_2: E_2^{0,2} \to E_2^{2,1})$. Since $F_2$ is an isomorphism on the source and target, the induced map $F_3: E_3^{0,2} \to \mathcal{E}_3^{0,2}$ is an isomorphism.
			Next, we have $E_4^{0,2} = \ker(d_3: E_3^{0,2} \to E_3^{3,0})$. We just showed $F_3$ is an isomorphism on the source $E_3^{0,2}$. We also need to show it is an isomorphism on the target, $E_3^{3,0}$.
			The term $E_3^{3,0}$ is the cokernel of $d_2: E_2^{1,1} \to E_2^{3,0}$. Since $F_2$ is an isomorphism on both of these terms, $F_3: E_3^{3,0} \to \mathcal{E}_3^{3,0}$ is an isomorphism.
			Now, for the $d_3$ differential, $F_3$ is an isomorphism on both the source and target. This implies the induced map on the kernels, $F_4: E_4^{0,2} \to \mathcal{E}_4^{0,2}$, is an isomorphism.
			Since later differentials are zero, we see that $E_\infty^{0,2} = E_4^{0,2}$. Thus, $F_\infty: E_\infty^{0,2} \to \mathcal{E}_\infty^{0,2}$ is an isomorphism.
		\end{itemize}
		This completes the proof of our lemma.
	\end{proof}

	\begin{proof}[Proof of Theorem \ref{main2}]
		Let $M$ be a finite discrete $p$-primary $G_{K,S}$-module. Then there exists an open normal subgroup $U=G_{K',S'}$ of $G_{K,S}$ such that $U$ acts trivially on $M$ and $ S'$ contains all primes of $K'$ that lie above the primes in $(S')_K$. By Theorem \ref{control}, there exists a finite set $S_{0}$ of primes of $K'$ such that 
		\begin{itemize}
			\item $S_{0} \cap (S'\cup \mathcal{T})= \emptyset$;
			\item $(S_{0}\cup S')_{\mathrm{fin}}\neq \emptyset$;
			\item the natural morphisms $\phi_{i}:H^{i}(G_{K',S'\cup S_{0}},M)\to H^{i}_{\et}(\operatorname{Spec}(\mathcal{O}_{K',S'\cup S_{0}}),\mathcal{M})$ are isomorphisms for all $i\leq 2$;
			\item $ S_{0}$ contains all primes of $K'$ that lie above the primes in $(S_0)_K$.
		\end{itemize}
		
		It follows that $G_{K',S'\cup S_0}$ is an open normal subgroup of $G_{K,S\cup (S_{0})_{K}}$, and hence we have the Hochschild--Serre spectral sequence
		\[  E^{r,s}_{2}=H^{r}(\operatorname{Gal}(K'/K),H^{s}(G_{K',S'\cup S_{0}},M))\Rightarrow H^{r+s}(G_{K,S\cup (S_{0})_{K}},M),\]
		where the group $G_{K',S'\cup S_{0}}$ acts trivially on $M$. On the other hand, since 
		\[ \operatorname{Spec}(\mathcal{O}_{K',S'\cup S_{0}})\to \operatorname{Spec}(\mathcal{O}_{K,S\cup (S_{0})_{K}}) \]
		is a Galois covering with Galois group $\operatorname{Gal}(K'/K)$, by \cite[Theorem 2.20 and Remark 2.21(b), Chapter III]{MR559531} we also have the Hochschild--Serre spectral sequence
		\begin{align*}
			\mathcal{E}_{2}^{r,s}=H^{r}(\operatorname{Gal}(K'/K),H^{s}_{\et}(\operatorname{Spec}&(\mathcal{O}_{K',S'\cup S_{0}}),\mathcal{M})) \\
			\Rightarrow &  H_{\et}^{r+s}(\operatorname{Spec}(\mathcal{O}_{K,S\cup (S_{0})_{K}}),\mathcal{M}).
		\end{align*}
		Moreover, the natural homomorphism (\ref{qiaoliang}) defines a morphism 
		\[ F:E^{r,s}\to \mathcal{E}^{r,s} \]
		of spectral sequences. Since $\phi_{i}$ are isomorphisms for all $i\leq 2$, $F$ is an isomorphism for all bigrades $(r,s) $ with $s \leq 2$ at the $E_2$ page. Finally, it follows from Lemma \ref{homologicalfact} that
		\[ H^{2}(G_{K,S\cup (S_{0})_{K}},M)\cong H^{2}_{\et}(\operatorname{Spec}(\mathcal{O}_{K,S\cup (S_{0})_{K}}),\mathcal{M}). \]
	   In addition, the maps 
	   \[ \phi_{i}:H^{i}(G_{K,S\cup (S_{0})_{K}},M)\to H^{i}_{\et}(\operatorname{Spec}(\mathcal{O}_{K,S\cup (S_{0})_{K}}),\mathcal{M}) \]
	   are always isomorphisms for $i=0,1$. This completes the proof of our theorem.
	\end{proof}

\subsection{Proof of Theorem \ref{main}}
Finally, we prove Theorem \ref{main} as follows. For parts (1) and (2), since the morphisms $\phi_{i}$ in \eqref{qiaoliang} are isomorphisms for $i=0, 1$, and injective for $i=2$, we obtain the inequality
	\[
	\chi_{2}(G_{K,S},M) \leq \chi_{2}(\operatorname{Spec}(\mathcal{O}_{K,S}),\mathcal{M}),
	\]
	where equality holds if and only if $\phi_{2}$ is an isomorphism. The claims then follow directly from Theorem \ref{calculationofetale}.
	
	Part (3) is an immediate consequence of Theorem \ref{main2} and Theorem \ref{calculationofetale}.

  \section{Proof of Theorem \ref{presentation}}\label{proofofsecondtheorems}
  In this section, our goal is to prove Theorem \ref{presentation}. To this end, we shall apply the results of Lubotzky \cite{MR1848964}. We begin by recalling the following definition. 
  
  \begin{Def}\cite[Definition 4.2]{MR1848964}
  	We say that a profinite group $G$ is a $d(G)$-abelian-indexed group if it is topologically finitely generated and, for every finite index subgroup $H$ of $G$, the abelianization of $H$ is isomorphic to $\widehat{\mathbb{Z}}^{m}$ with $m=1+(d(G)-1)[G:H]$, where $d(G)$ denotes the minimal number of topological generators of $G$.
  \end{Def}
  
 For any prime number $p$ and any finite $\mathbb{F}_{p}[[G]]$-module $M$, we define
   \[    	
 v(G,p,M) := \left\lceil \dfrac{\dim H^{2}(G,M) - \dim H^{1}(G,M)+\dim H^{0}(G,M)}{\dim M} \right\rceil,
 \] 
  where $\lceil x \rceil$ denotes the ceiling function, and $\mathbb{F}_{p}[[G]]$ is the completed group algebra of $G$ over $\mathbb{F}_p$. In \cite[Theorem 0.2]{MR1848964}, Lubotzky proved the following theorem.
  
  \begin{Thm}\label{Lubotzky}
  	Let $G$ be a topologically finitely generated profinite group. If $G$ is not $d(G)$-abelian-indexed, then
  	\[  r(G)-d(G)=\max_{p,M}v(G,p,M) -1,\]
    where the maximum is taken over all prime numbers $p$ and all finite simple $\mathbb{F}_{p}[[G]]$-modules $M$.
  \end{Thm}
  
  To apply this theorem to our group $G = G_{K,S}$, we need to establish the following lemma.
  
  \begin{Lem}
  For any number field $K$ and any finite set $S$ of primes of $K$, the profinite group $G_{K,S}$ is not $d(G_{K,S})$-abelian-indexed.
  \end{Lem}
  \begin{proof}
  	We may assume that $G_{K,S}$ is topologically finitely generated. Choose a prime number $p$ such that $S \cap S_{p} = \emptyset$. By class field theory, the group $G_{K,S}^{\mathrm{ab}}(p)$ is finite. On the other hand, if $G_{K,S}$ were a $d(G_{K,S})$-abelian-indexed group, then it would imply that $G_{K,S}^{\mathrm{ab}}(p) \cong \mathbb{Z}_{p}^{d(G_{K,S})}$, which contradicts the finiteness of $G_{K,S}^{\mathrm{ab}}(p)$. This completes the proof.
  \end{proof}
  
  For any finite $\mathbb{F}_{p}[[G_{K,S}]]$-module $M$, we define
  \[ \dim \chi_{2}(G_{K,S},M):= \log_{p} \chi_{2}(G_{K,S},M). \]
   As a corollary, we immediately obtain
  \begin{equation}\label{formula}
   r(G_{K,S})-d(G_{K,S})=\max_{p,M} \left\lceil \dfrac{\dim \chi_{2}(G_{K,S},M)}{\dim M} \right\rceil -1.
  \end{equation}
 
 Let $\mu_p$ denote the $G_K$-module of $p$-th roots of unity in $\overline{K}^{\times}$, and let $M'$ be the Cartier dual of $M$. We need the following result.
 
 \begin{Lem}\label{observation}
 	Let $M$ be a finite simple $\mathbb{F}_p[[G_K]]$-module. Then $\dim H^0(G_K, M') = 1$ if $M \cong \mu_p$ as $\mathbb{F}_p[[G_K]]$-modules, and is $0$ otherwise.
 \end{Lem}
 
 \begin{proof}
 	By the definition of the Cartier dual, there is a natural isomorphism of $\mathbb{F}_p$-vector spaces
 	\[
 	H^0(G_K, M') \cong \operatorname{Hom}_{\mathbb{F}_p[[G_K]]}(M, \mu_p).
 	\]
 	Since $\mu_p$ and $M$ are both simple, Schur's lemma implies that $\operatorname{Hom}_{\mathbb{F}_p[[G_K]]}(M, \mu_p)$ is $1$-dimensional over $\mathbb{F}_p$ if $M \cong \mu_p$, and vanishes otherwise.
 \end{proof}

Now, let $M$ be a finite simple $\mathbb{F}_p[[G_{K,S}]]$-module. We evaluate the upper bounds by distinguishing two cases. If either $S_{\mathrm{fin}} \neq \emptyset$ or $M \not\cong \mu_p$, then Theorem \ref{main}, Lemma \ref{observation}, and the product formula \cite[Proposition 1.3, Chapter III]{MR1697859} yield
\begin{align*}
	\dfrac{\dim \chi_{2}(G_{K,S},M)}{\dim M} &\leq \dfrac{\sum_{v\in S_{\infty}}\dim H^{0}(G_{v},M)-[K:\mathbb{Q}]\dim M+ \sum_{v \in S_p \setminus S} [K_v:\mathbb{Q}_p]\dim M}{\dim M} \\
	&\leq \dfrac{\sum_{v\in S_{\infty}}\dim H^{0}(G_{v},M)-[K:\mathbb{Q}]\dim M+ \sum_{v \in S_p } [K_v:\mathbb{Q}_p]\dim M}{\dim M} \\
	&= \dfrac{\sum_{v\in S_{\infty}}\dim H^{0}(G_{v},M)}{\dim M} \\
	&\leq \mathcal{R}(K).
\end{align*}

In the remaining case where $S_{\mathrm{fin}} = \emptyset$ and $M \cong \mu_p$, we obtain the bound
\[ 	\dfrac{\dim \chi_{2}(G_{K,S},M)}{\dim M}\leq \sum_{v\in S_{\infty}}\dim H^{0}(G_{v},\mu_{p})+1\leq \mathcal{R}(K)+1.
 \]
Substituting these bounds into \eqref{formula}, we conclude that
\[ r(G_{K,S})-d(G_{K,S})\leq \mathcal{R}(K)-\gamma. \]

   Furthermore, by \cite[Corollary 2.5]{MR1848964}, the profinite group $G_{K,S}$ admits a finite presentation on $d(G_{K,S})$ generators and $\mathcal{R}(K)+d(G_{K,S})-\gamma$ relations. 
   
   To see that $r(G_{K,S})-d(G_{K,S})\geq 0$, we choose a prime number $p$ such that $S\cap S_{p}=\emptyset$ and consider the trivial $G_{K,S}$-module $\mathbb{F}_{p}$. Since the abelianization of $G_{K,S}(p)$ is finite by class field theory, we have
   \[ \dim H^{2}(G_{K,S}(p),\mathbb{F}_{p})-\dim H^{1}(G_{K,S}(p),\mathbb{F}_{p})\geq 0. \]
 
   Since 
   \begin{align*}
   	\dim H^{1}(G_{K,S},\mathbb{F}_{p})&=\dim  H^{1}(G_{K,S}(p),\mathbb{F}_{p}),\\
   	\dim H^{2}(G_{K,S},\mathbb{F}_{p})&\geq \dim H^{2}(G_{K,S}(p),\mathbb{F}_{p}),
   \end{align*}
  we have
     \begin{align*}
   r(G_{K,S})-d(G_{K,S})	\geq &   \dim H^{2}(G_{K,S}(p),\mathbb{F}_{p})-\dim H^{1}(G_{K,S}(p),\mathbb{F}_{p})\\
  	\geq & 0,
  \end{align*}
   which proves the first assertion.
   
   To prove the second assertion, we choose an odd prime number $q$ such that $S_q \cap S = \emptyset$, and consider the trivial $G_{K,S}$-module $\mathbb{F}_q$. By Theorem \ref{main}, there exists a finite set $S_{0}$ of primes of $K$ such that $S_{0}\cap S_{q}=\emptyset,(S_{0}\cup S)_{\mathrm{fin}}\neq \emptyset$, and 
    \[ 	\dfrac{\dim \chi_{2}(G_{K,S\cup S_{0}},\mathbb{F}_{q})}{\dim \mathbb{F}_{q}}=\sum_{v\in S_{\infty}}\dim H^{0}(G_{v},\mathbb{F}_{q})= \mathcal{R}(K).\]
    By the first assertion and (\ref{formula}), we have
    \[  r(G_{K,S\cup S_{0}})-d(G_{K,S\cup S_{0}})=\mathcal{R}(K)-1. \]
 This establishes the second assertion, completing the proof of the theorem.
    
\section{An example}\label{examples}
If $K$ is a number field, $S$ is a finite set of primes containing all archimedean primes and $M$ is a finite discrete $G_{K,S}$-module, then we will denote by $\widetilde{\chi}_{2}(G_{K,S},M)$ the right hand side of (\ref{euler}). In this section, we provide an example where
\[ \chi_{2}(G_{K,S},M) < \widetilde{\chi}_{2}(G_{K,S},M), \]
which also illustrates the strict inequality
\[r(G_{K,S}) - d(G_{K,S}) < \mathcal{R}(K) - \gamma.  \]

 \begin{example}\label{anexample}
 Let $K=\mathbb{Q}(\sqrt{-120})$. By \cite{MR1617407}, $G_{K,S_{\infty}}$ is isomorphic to the quaternion group $Q_{8}$ of order $8$. For any prime number $p$, consider $\mathbb{F}_{p}$ as a trivial $G_{K,S_{\infty}} $-module. Then we have
 \[ 	\chi_{2}(G_{K,S_{\infty}},\mathbb{F}_{p}) =\dfrac{[H^{0}(Q_{8},\mathbb{F}_{p})]\cdot [H^{2}(Q_{8},\mathbb{F}_{p})]}{[H^{1}(Q_{8},\mathbb{F}_{p})]}=p. \]
 Since the field $K = \mathbb{Q}(\sqrt{-120})$ contains a primitive $p$-th root of unity if and only if $p=2$, we have
 	\begin{equation*}
 	\widetilde{\chi}_{2}(G_{K,S_{\infty}},\mathbb{F}_{p})=\begin{cases}
 		p,&\text{if } p\geq 3,\\
 		p^2,&\text{if } p=2.
 	\end{cases}
 \end{equation*}
 Thus, taking $p=2$ gives the desired strict inequality
 \[ \chi_{2}(G_{K,S_{\infty}},\mathbb{F}_{2}) = 2 < 4 = \widetilde{\chi}_{2}(G_{K,S_{\infty}},\mathbb{F}_{2}). \]
  
 Furthermore, since $Q_{8}$ is a $2$-group, it follows from \cite[Corollary 5.5]{MR1848964} that
 \[r(G_{K,S_{\infty}}) =r(Q_{8})=\dim_{\mathbb{F}_{2}}H^{2}(Q_{8},\mathbb{F}_{2})=2, \]
   which yields
   \[ r(G_{K,S_{\infty}})-d(G_{K,S_{\infty}})=2-2<1-0=\mathcal{R}(K)-\gamma.  \]
 \end{example}

\section{Proofs of Theorems \ref{counterexampletodimensionalconj} and \ref{corofUFM}}\label{proofofothertheorems}

This section is devoted to the proofs of Theorems \ref{counterexampletodimensionalconj} and \ref{corofUFM}. We begin with a simple lemma.

\begin{Lem}\label{linearalgebra}
	Let $G$ be a finite group such that either $G = \{1\}$ or $G \cong \mathbb{Z}/2\mathbb{Z}$, and let $\rho: G \to \mathrm{GL}_{n}(F)$ be a representation over a field $F$. If $\mathrm{ad} = \mathrm{ad}(\rho)$ is the adjoint representation of $\rho$, then we have
	\[ \dim_{F} H^0(G, \mathrm{ad}) \geq n. \]
\end{Lem}

\begin{proof}
	If $G = \{1\}$, then $H^{0}(G, \mathrm{ad})$ is the space $M_{n}(F)$ of $n \times n$ matrices over $F$, and hence $\dim_{F} H^{0}(G, \mathrm{ad}) = n^{2} \geq n$. Thus, we may assume that $G \cong \mathbb{Z}/2\mathbb{Z}$. Let $c$ be the generator of $G$, and let $C = \rho(c) \in \mathrm{GL}_{n}(F)$. Then $H^0(G, \mathrm{ad})$ is precisely the centralizer $Z_{M_n(F)}(C)$ of $C$ in $M_{n}(F)$. It is a standard fact in linear algebra that the dimension of the centralizer of any $n \times n$ matrix over an arbitrary field is at least $n$. Thus, we have
	\[ \dim_{F} H^{0}(G, \mathrm{ad}) = \dim_{F} Z_{M_n(F)}(C) \geq n. \]
\end{proof}

We are now ready to prove Theorems \ref{counterexampletodimensionalconj} and \ref{corofUFM} as follows.

\begin{proof}[Proof of Theorem \ref{counterexampletodimensionalconj}]
	By Theorem \ref{main}, there exists a finite set $T \supset S$ such that $T \cap S_p = \emptyset$, $T \supset S_{\infty}$ and
	\[ \dim_{\mathbb{F}} \chi_{2}(G_{K,T},\mathrm{ad}):=\log_{q} \chi_{2}(G_{K,T},\mathrm{ad})=\sum_{v\in S_{\infty}}\dim_{\mathbb{F}}H^{0}(G_{v},\mathrm{ad}),\]
	where $q$ is the cardinality of $\mathbb{F}$. Since $\overline{\rho}$ is absolutely irreducible as a representation of $G_{K,S}$, it remains absolutely irreducible as a representation of $G_{K,T}$. Therefore, we have $h^{0}(G_{K,T}, \mathrm{ad}) = h^{0}(G_{K,S}, \mathrm{ad}) = 1$. Furthermore, since $n \geq 2$ and $S_{\infty}$ is nonempty, Lemma \ref{linearalgebra} implies that $\sum_{v \in S_{\infty}} \dim_{\mathbb{F}} H^{0}(G_{v}, \mathrm{ad}) \geq n \geq 2$. Thus, we obtain
	\begin{align*}
		h^{1}(G_{K,T}, \mathrm{ad}) - h^{2}(G_{K,T}, \mathrm{ad}) &= -\sum_{v \in S_{\infty}} \dim_{\mathbb{F}} H^{0}(G_{v}, \mathrm{ad}) + h^{0}(G_{K,T}, \mathrm{ad}) \\
		&= -\sum_{v \in S_{\infty}} \dim_{\mathbb{F}} H^{0}(G_{v}, \mathrm{ad}) + 1 \\
		&\leq -2 + 1 \\
		&< 0.
	\end{align*}
	This completes the proof.
\end{proof}

\begin{proof}[Proof of Theorem \ref{corofUFM}]
	Recall from \cite[Theorem 4.2]{MR1860043} that the universal deformation ring $R_{\overline{\rho}}$ of $\overline{\rho}$ has a presentation
	\[ W(\mathbb{F})[[X_{1}, \dots, X_{h^{1}}]] / (f_{1}, \dots, f_{h^{2}}), \]
	where $h^{i} = h^{i}(G_{K,S}, \mathrm{ad})$ for $i=1, 2$. Under our assumptions, \cite[Theorem 1]{MR3294389} ensures that $R_{\overline{\rho}}$ is finite over $W(\mathbb{F})$. 
	
	Suppose toward a contradiction that $h^{1} - h^{2} \geq 0$. Then, an application of \cite[Lemma 2, Appendix]{MR2004460} implies that $R_{\overline{\rho}}$ is finite flat over $W(\mathbb{F})$. Consequently, the ring $R_{\overline{\rho}}$ admits a characteristic zero point; that is, there exists a $p$-adic representation $\rho: G_{K,S} \to \mathrm{GL}_{n}(\overline{\mathbb{Q}}_{p})$ whose reduction modulo $p$ is $\overline{\rho}$. However, since the image of $\overline{\rho}$ contains $\mathrm{SL}_{n}(\mathbb{F})$, the image of $\rho$ is necessarily infinite by \cite[Main theorem]{MR3336600}. This yields a contradiction to the unramified Fontaine--Mazur conjecture, thereby completing the proof.
\end{proof}

  \section*{Acknowledgements}
 The author is grateful to Adrien Morin for discussions concerning \cite[Proposition 6.23]{MR4502240}. The author would also like to thank the anonymous referee for their valuable comments and suggestions, which greatly improved the quality of this paper.

\bibliographystyle{plain}
\bibliography{euler-char}

\end{document}